\documentclass[a4paper,10pt]{amsart}
\usepackage[utf8]{inputenc}
\usepackage[cmtip,arrow]{xy}
\usepackage{lscape}
\usepackage{latexsym}
\usepackage[activeacute,english]{babel}
\usepackage{graphicx}
\usepackage{amsmath}
\usepackage{amsfonts}
\usepackage{amsthm}
\usepackage{amssymb}
\usepackage{amscd}
\usepackage{mathrsfs}
\usepackage{pb-diagram,pb-xy}
\usepackage{tikz}
\xyoption{all}
\newtheorem{theorem}{Theorem}[section]
\newtheorem{proposition}[theorem]{Proposition}
\newtheorem{lemma}[theorem]{Lemma}
\newtheorem{definition}[theorem]{Definition}

\begin{document}

\title[Linear Filters]{Linear Filters and Hereditary Torsion Theories in Funtor Categories}
\author{M. Ortiz-Morales$^{a}$}
\address{Facultad de Ciencias UAEM\'ex, Mexico}
\email{mortizmo@uaemex.mx}
\thanks{}
\author{S. Diaz-Alvarado$^{b}$}
\address{Facultad de Ciencias UAEM\'ex, Mexico}
\email{sda@uaemex.mx}
\urladdr{http://www.uaemex.mx}
\thanks{$^{A}$,$^{B}$ Universidad Aut\'onoma del Estado de M\'exico, UAEM\'ex. The firs author thanks PROMEP for giving him financial support
to to this work. }
\date{November 28, 2014}
\dedicatory{}
\thanks{This paper is in final form and no version of it will be submitted
for publication elsewhere.}

\begin{abstract}
we introduce the notion of Gabriel filter for a preadditive category $\mathcal{C}$ and we show that there is a
 bijective correspondence between  Gabriel  filters of  $\mathcal{C}$ and hereditary torsion theories in the category of additive functors $(\mathcal{C},\mathbf{Ab})$, obtaining a generelization of the theorem given by Gabriel [Ga] and Maranda [Ma] which establishes a bijective correspondence between Gabriel
 filters for a ring and hereditary torsion theories in the corresponding category of modules. 
\end{abstract} 
\maketitle

\section{Introduction and Basic Results}
In this paper $\mathcal{C}$ will be an arbitrary skeletally small preadditive category, and
$\mathrm{Mod}(\mathcal{C} )$ will denote the category of contravariant functors from $\mathcal{C}$ to the category
of abelian groups $\mathbf{Ab}$. Following the approach by Mitchel [M], we can think of $\mathcal{C}$ as
a ring “with several objects” and $\mathrm{Mod}(\mathcal{C} )$ as a category of $\mathcal{C}$ -modules. The aim of
the paper is to show that the notions of linear filter  can be extended to preadditive categories 
obtaining generalizations of the  theorem due to Gabriel that stablishes a bijective correspondence 
hereditary torsion theories and linear filters. 
The notion of torsion theory (or torsion pair) was introduced by S. E.
Dickson [Di] in the sixties in the setting of abelian categories, generalizing
the classical notion for abelian groups. Since then it has received a lot of
attention in various contexts, like non commutative localization, or representation theory of artin algebras.
Of particular importance are the hereditary torsion pairs, due to their
role in localization theory, notably their characterization in terms of Gabriel
filters of ideals and in terms of Gabriel topologies.
\subsection{Torsion Theories}

In this part we recall some basic concepts about torsion theories. The content of this subsection is  taken from [Bo].

  Let $\mathcal{A}$ an abelian category. 
  
  \begin{definition}
   A torsion theory for $\mathcal A$ is a pair $(\mathcal{T,F})$ of classes of objects on $\mathcal{A}$ such that
   \begin{itemize}
    \item [(i)] $\mathrm{Hom}(T,F)=0$ for all $T\in\mathcal{T}, F\in\mathcal{F}$.
   \item [(ii)] If $\mathrm{Hom}(C,F)=0$ for all $F\in\mathcal{F}$, then $C\in\mathcal{T}$.
   \item [(iii)] If $\mathrm{Hom}(T,C)=0$ for all $T\in\mathcal{T}$, then $C\in\mathcal{F}$.
   \end{itemize}
  \end{definition}

$\mathcal T$ is called a torsion class and its objects are torsion objects, while $\mathcal{F}$ is a torsion-free class
 consisting  of torsion free objects.
 
 Any given class $\mathcal{B}$ of objects \textit{generates} a torsion theory in the following way:
 \begin{eqnarray*}
  \mathcal{F}=\{F|\mathrm{Hom}(C,F)=0\ \text{for all $C\in\mathcal{B}$}\}.\\
    \mathcal{T}=\{T|\mathrm{Hom}(T,F)=0\ \text{for all $F\in\mathcal{F}$}\}.
\end{eqnarray*}
 
 Clearly this pair $(\mathcal{T,F})$ is a torsion theory, and $\mathcal{T}$  is the smallest torsion class containing $\mathcal{B}$.
 Dually, the class $\mathcal{B}$ \textit{cogenerates} a torsion theory $(\mathcal{T,F})$ such that $\mathcal{F}$ is the smallest torsion free theory

 A torsion pair $(\mathcal{T,F})$ is \textit{hereditary} if $\mathcal{T}$ is closed under subojects. In case
$\mathcal{A}$  is a module category, the hereditary torsion pairs are in bijective correspendence with the
Gabriel filters of  the ring (see [Bo. Theorem 5.1, Ch. VI]).

A class $\mathcal{B}$ of objects is called a pretorsion class if it is closed under quotients objects and coproducts
and it is a pretorsion-free class if it is closed under subobjects and products. A pretorsion class is hereditary if it is closed
under subobjects.

In the first section we fix the notation and recall some notions from functor
categories that will be used throughout the paper. In the second section we define 
the concept of linear filter for a preatidive category $\mathcal{C}$ and we  generalize the theorem
given in [Bo. Theorem 5.1, Ch. VI]  and we  stablish a one to one correspondence  between
Gabriel Filters in a preadditive category and hereditary torsion theories 
in the category of $\mathcal{C}$-modules.  In the third section we see
 that the linear filters induce a topology for $\mathcal{C}(A,B)$ for every pair
 of objects $A,B$ in $\mathcal{C}$ that makes that the
 composition $\mathcal{C}(A,B)\times\mathcal{C}(B,C)\rightarrow \mathcal{C}(A,C)$  be continous, obtaining
 similar results as the given in [Bo, Ch. VI]. Finally, in the  fourth section  we explore 
 some examples of linear filters for path categories, and we show how some natural examples of
 hereditary torsion theories appear in  the category of $\mathcal C$-modules.

\subsection{The Category of $\mathcal{C}$-modules} 
Let $\mathcal C$ be a preadditive skeletally small category. By $\mathrm{Mod}(\mathcal C )$ we denote the cate
gory of additive contravariant functors from $C$ to the category of abelian groups.
$\mathrm{Mod}(\mathcal C )$ is then an abelian category with arbitrary sums and products; in fact it has
arbitrary limits and colimits, and the filtered limits are exact (Ab5 in Grothendiek
terminology). It has enough projective and injective objects. For any object $C \in\mathcal  C$ ,
the representable functor $\mathcal C ( , C)$ is projective, the arbitrary sums of representable
functors are projective, and any object $M \in\mathrm{ Mod}(\mathcal C )$ is covered by an epimorphism
$\coprod_{i\in I} \mathcal{C}(-,C_i)\rightarrow M\rightarrow 0$ (see [AQM]).

We will indistinctly say that $M$ is an object of $\mathrm{Mod}(\mathcal C )$ or that $M$ is a $\mathcal{C}$ -module. A
representable functor $\mathcal C ( -, C)$ will sometimes be denoted by $( -, C)$.

Let $M$ be  a $\mathcal C$-module and $C$ and object in $\mathcal{C}$. Then, by Yoneda's Lemma there exists a one to one correspondence
\[
\theta=\theta_{C,M}:(\mathcal{C}(-,C),M)\rightarrow M
\]
where $(\mathcal{C}(-,C),M)$ is the class of natural transformations from the funtor $\mathcal C(-,C)$ to the set valued funtor $M$, given by
$ \theta(\eta)=\eta_C(1_C)$.

Since $\mathcal C$ is a skeletally small category, then the class  $(\mathcal{C}(-,C),M)$ 
is a set and we can index it  as
 $(\mathcal{C}(-,C),M)=\{\eta^x:\mathcal{C}(-,C)\rightarrow M\}_{x\in M(C)}$, and it induces a morphism of $\mathcal{C}$-modules
 $\eta^C=(\eta^x)_{x\in M(C)}:\coprod_{x\in M(C)}\mathcal{C}(-,C)\rightarrow M$. It is clear that
  $\eta^C_C=(\eta^x_C)_{x\in M(C)}:\mathcal{C}(C,C)\rightarrow M(C)$ is an epimorphism, and we have an epimorphism of 
  $\mathcal{C}$-modules
  \[
\eta=(\eta^C)_{C\in\mathcal{C}}:\coprod_{C\in\mathcal{C}}\coprod_{x\in M(C)}\mathcal{C}(-,C)\rightarrow M
  \]
  
 As above,  let $x\in M(C)$ and $\eta^x:\mathcal{C}(-,C)\xrightarrow {p_x}\mathrm{Im}(\eta^x)\xrightarrow{j^x} M$ the canonical factorization of $\eta^x$. 
 The family $\{j^x:\mathrm{Im}(\eta^x)\rightarrow M\}$ induces a morphism of $\mathcal C$-modules
  $j^C=(j^x)_{x\in M(C)}:\coprod\mathrm{Im}(\eta^x)\rightarrow M$.  
  Then $j^C_C=(j_C^x)_{x\in M(C)}:\coprod\mathrm{Im}(\eta^x_C)\rightarrow M(C)$
is  an isomorphism and we have an isomorphism of $\mathcal{C}$-modules 
 \[
j=(j^C)_{C\in\mathcal{C}}:\coprod_{C\in\mathcal{C}}\coprod_{x\in M(C)}\mathrm{Im}(\eta^x)\rightarrow M
  \]
Finally, we can put the above isomorphism as 
\begin{equation}\label{isocop}
 \coprod_{C\in\mathcal{C}}\coprod_{x\in M(C)}\mathcal{C}(-,C)/\mathrm{Ker}(\eta^x)\xrightarrow{\cong} M
\end{equation}

  \subsection{The funtor $D:\mathrm{Mod}(\mathcal{C})\rightarrow \mathrm{Mod}(\mathcal{C}^{op})$}
  We consider the funtor $ D:\mathrm{Mod}(\mathcal{C})\rightarrow \mathrm{Mod}(\mathcal{C}^{op})$ 
   defined by $(DM)(C)=\mathrm{Hom}_{\mathbf{Ab}}(M(C),\mathbf{Q/Z})$ for all pair $M\in\mathrm{Mod}(\mathcal C)$ and $C\in\mathcal{C}$.
   
   We left to the reader the  following
  \begin{proposition}\label{dual}
   Let $M\in\mathrm{Mod}(\mathcal C)$ and $N\in\mathrm{Mod}(\mathcal C^{op})$. Then we have an isomorphism
   \begin{eqnarray*}
    \Phi:\mathrm{Hom}_{\mathcal C}(M,DN)\rightarrow \mathrm{Hom}_{\mathcal C^{op}}(N,DM)
   \end{eqnarray*}
defined  by $[\Phi(\eta)]_{C}(Y)(X)=\eta_{C}(Y)(X)$ for each  pair $X\in M(C)$ and $Y\in N(C)$.
\end{proposition}

Now we can prove the following 
\begin{proposition}\label{injective}
Let $C\in\mathcal{C}$. Then the $\mathcal{C}$-module $D(\mathcal{C}(C,-))$ is injective
\end{proposition}
\begin{proof}
 It follows from the Yoneda's Lemma and the Proposition \ref{dual}.
\end{proof}

\subsection{Ideals}
A right ideal in an additive category $\mathcal{C}$ is a subfuntor of $\mathcal{C}(-,C)$ for some $C\in\mathrm{\ Ob}\ \mathcal{C}$, and  a left ideal is a 
subfuntor of $\mathcal{C}(C,-)$ (see [M]and [BR]). A two sided ideal is a subfuntor of the two variable funtor $\mathcal{C}(-,?)$.  Given a two sided ideal $I$ in $\mathcal{C}$
we can form the quotient category  $\mathcal{C}/I$ wich has the same objects  as $\mathcal{C}$ and
\[
(\mathcal{C}/I)(A,B)=\mathcal{C}(A,B)/I(A,B)
\]

Given a family $\mathcal{B}$ of objects in $\mathcal{C}$ we can define a two sided ideal $I_{\mathcal{B}}$ of $\mathcal{C}$ generated  by $\mathcal{B}$ setting $f\in I_{\mathcal{B}}(A,B)$ if and only if  $f $ is a finite sum of maps of the form $A\xrightarrow{h} C \xrightarrow{g}B$ with $C$ in $\mathcal{B}$.

 If $I$ is  a two sided ideal in $\mathcal{C}$ and $M$ is a $\mathcal{C}$-module the  $\mathcal{C}$-submodule $IM$ is defined as
 \[
 IM(A)=\sum_{f\in I(A,C), C\in \mathcal{C}} \mathrm{Im} M(f)
 \]

It is clear that $IM$ is a subfuntor of $M$. Morever any epimorphism $M\rightarrow M'$ induces an epimorphism $IM\rightarrow IM'$ .

Note that for every morphism of $\mathcal{C}$-modules  $\eta:\mathcal{C}(-,C)\rightarrow M$, the funtor kernel 
$\mathrm{Ker}(\eta)\subset \mathcal{C}(-,C)$ is a right ideal.

\section{Linear Filters and hereditary torsion theories}

In this section we introduce the notions of linear filter and Gabriel filter for an arbitrary preadditive category $\mathcal C$ and we show that there is
a one to one correspendence between linear filters in $\mathcal{C}$ and classes of pretorsion theories in $\mathrm{Mod}(\mathcal{C})$
also a one to one correspendence between Gabriel filters and classes of hereditary torsion theories. 
\subsection{Linear Filters}
Let be $N$ a $\mathcal{C}$-module,  $K$ be a $\mathcal{C}$-submodule of $N$ and $C$ an object in $\mathcal{C}$. For each 
$C'\in \mathcal{C}$ 
 and each $x \in N(C)$  we consider the set \[
(K(C'):x)=\{f\in\mathcal{C}(C',C)|N(f)(x)\in K(C')\}
\]

Thus, we can define a $\mathcal{C}$-module $(K(-):x):\mathcal{C}\rightarrow \mathbf{Ab}$ as $(K(-):x)(C')=(K(C'):x)$ 
for all $C'\in \mathcal{C}$ . Clearly $(K(-):x)$ is an ideal of $\mathcal{C}(-,C)$. 

Let  $\mathbf{0}$  be  the  trivial $\mathcal{C}$-module, $\mathbf{0}(C)=0_{\mathbf{Ab}}$, for all $C\in\mathcal{C}$.
We denote by $\mathrm{Ann}(x,-)$ the functor $(\mathbf{0}(-):x)$ which is defined as $\mathrm{Ann}(x,-)(C')=
\mathrm{Ann}(x,C')=(\mathbf{0}(C'):x)=\{f\in\mathcal{C}(C',C)|N(f)(x)\in \mathbf{0}(C')\}=\{f\in\mathcal{C}(C',C)|N(f)(x)=0\}$.

We have the following Lemma

\begin{lemma}\label{firstlemma}
\begin{itemize}
\item[(i)]  Let $N$ and $K$ be  $\mathcal{C}$-modules such that $K\subset N$.  Then,  $\mathrm{Ann}(x+K(C),-)= (K(-):x)$  for all $C\in\mathcal{C}$ and $x\in N(C)$.
\item[(ii)] Let $C\in\mathcal{C}$ and  $I$ be and ideal of $\mathcal{C}(-,C)$. Then, $(I(-):1_C)=I$, for the dentity morphism  $1_C\in\mathcal{C}(C,C)$.
\end{itemize}
\end{lemma}
\begin{proof}
 We left the proof for the reader.
\end{proof}

Now we can define the concept of linear filter for a preadditive category. This definition generalizes the definition of linear filter
for rings.

\begin{definition}  1) A family $\mathcal{F}$ of ideals of $\mathcal{C}(-,C)$ is a  filter for $\mathcal{C}(-,C)$ if the following conditions hold:
\begin{itemize}
\item[($T_1$)] If $I\in\mathcal{F}$ and $I\subset J$ then $J\in\mathcal{F}$
\item[($T_2$)]  If $I$ and $J$ belong to $\mathcal{F}$, then $I\cap J\in\mathcal{F}$
\end{itemize}
2) A collection $\{\mathcal{F}_C\}_{C\in \mathcal{C}}$ is a linear filter for the category $\mathcal{C}$ if  each   
$\mathcal{F}_{C}$ is a filter for
 $\mathcal{C}(-,C)$ and:
 \begin{itemize}
  \item[($T_3$)] For all  $I\in \mathcal{F}_{C}$, $B\in\mathcal{C}$ and  each $h\in \mathcal{C}(B,C)$ we have 
 $(I(-):h)\in\mathcal{F}_{B}$, where $(I(C'):h)=\{f:C'\rightarrow  B|\mathcal{C}(f,A)(h)=hf\in I(C')\}$ for all $C'\in \mathcal{C}$.
 \end{itemize}

 3) A collection   $\{\mathcal{F}\}_{C\in\mathcal{C}}$ is a Gabriel  filter for the category $\mathcal{C}$ if the following holds:
 \begin{itemize}
  \item[($T_4$)]
Let $J\in\mathcal{F}_C$ and assume that $I\subset\mathcal{C}(-,C)$ is  an ideal such that $(I(-):h)\in\mathcal{F}_B$ for all 
$h\in J(B)$  for all $B\in\mathcal{C}$, then $I\in\mathcal{F}_C$. 
\end{itemize}
\end{definition}

Now we show that there is a   bijection between hereditary
pretorsion classes and linear filters on $\mathcal C$. To do it we need two previous lemmas.

\begin{lemma}
Let $\mathcal{F}=\{\mathcal{F}_C\}$ be  a linear filter on  $\mathcal{C}$. Then,  the class $\mathcal{T}_{\mathcal{F}}$
consisting of  $\mathcal{C}$-modules  $M$  for which  $\mathrm{Ann}(x,-)\in \mathcal{F}_C$, for all
$x\in M(C)$ and $C\in \mathcal{C}$, is a hereditary pretorsion class.
\end{lemma}
\begin{proof}
 1)  $\mathcal{T}_{\mathcal{F}}$ is closed under subobjects. Let $N\in \mathcal{T}_\mathcal{F}$ and $K$   be a $\mathcal{C}$-submodule of  $N$. 
It follows that  for all  $C \in \mathcal{C}$ and $x\in N(C)$ the ideal  $\mathrm{Ann}(x,-)$ lies $ \mathcal{F}_C$.  Then $\mathrm{Ann}(x,-)\in \mathcal{F}_C$ 
for all  $x\in K(C)$ since $K(C)\subset N(C)$.

2) $\mathcal{T}_{\mathcal{F}}$ is closed under quotients. Let $N\in \mathcal{T}_{\mathcal{F}}$ and $K$ be $\mathcal{C}$-submodule of $N$. 
Then, for all $C\in\mathcal{C}$ and $x+K(C)\in \displaystyle\frac{N}{K}(C)=\displaystyle\frac{N(C)}{K(C)}$ we have by
Lemma \ref{firstlemma} that
$\mathrm{Ann}(x+K(C),-)=(K(-):x)\in\mathcal{F}_C$ by $T_2$. 

3) $\mathcal{T}_{\mathcal{F}}$ is closed under coproducts. Let $\{N_\lambda\}_{\lambda\in N}$ be a family of
$\mathcal{C}$-modules,
$N_\lambda\in \mathcal{T_F}$, and $\overline{X}=(x_\lambda)_{\lambda\in N}\in (\coprod_{\lambda}N_{\lambda})(C)=\coprod_{\lambda}
N_{\lambda}(C)$.  If we consider the set
$\{x_{\lambda_1},\ldots,x_{\lambda_n}\}$ consisting of all the non-zero coordinates of $\overline{X}$ then  
we have $\mathrm{Ann}(\overline{X},-)$$=\displaystyle{\cap^n_{i=1}}\mathrm{Ann}(x_{\lambda_i},-)\in \mathcal{F}_C$. This implies that $\coprod_\lambda N_\lambda\in \mathcal{T_F}$.
\end{proof}

\begin{lemma}
Let $\mathcal{T}$ be a hereditary pretorsion class. For each $C\in  \mathcal{C}$, let 
$\mathcal{F}_C=\{I\subset \mathcal{C}(-,C)\}$  be the family of right ideals for which
$\displaystyle\frac{\mathcal{C}(-,C)}{I}\in \mathcal{T}$. Then the class 
$\mathcal{F}_{\mathcal T}=\{\mathcal{F}_C\}_{C\in \mathcal{C}}$ is a linear filter on $\mathcal{C}$.

\end{lemma}

\begin{proof}
 ($T_1$) Let $I\in \mathcal{F}_C$ and $J\subset\mathcal{C}(-,C)$ a right ideal such that $I\subset J$. The monomorphism 
 \begin{eqnarray*}
 \varphi:\frac{\mathcal{C}(-,C)}{J} &\rightarrow& \frac{\mathcal{C}(-,C)}{I}\\
 \ \varphi_{C'}(f+J(C'))&=&f+I(C')
 \end{eqnarray*}
 for all  $ f\in\mathcal{C}(C',C)$ and $C'\in\mathcal{C}$,  implies that  
 $\displaystyle\frac{\mathcal{C}(-,C)}{J}\in \mathcal{T}$ since $\mathcal{T}$ is closed under subobjects. Hence $J\in \mathcal{F}_C$.

($T_2$) If $I,J\in \mathcal{F}_C$ then $I\cap J\in \mathcal{F}_C$ beacuse we have the monomorphism
\[
 \varphi:\displaystyle\frac{\mathcal{C}(-,C)}{I\cap J}\rightarrow \displaystyle\frac{\mathcal{C}(-,C)}{I}\coprod  \displaystyle\frac{\mathcal{C}(-,C)}{J}
\]
 given by
$\varphi_{C'}(f+(I\cap J)(C'))=(f+I(C'), f+J(C'))$, $C'\in \mathcal{C}$, $f\in \mathcal{C}(C',C)$ since $\mathcal{T}$
is a closed under coproducts. 

($T_3$) Let $I\in \mathcal{F}_C$. Then $\displaystyle\frac{\mathcal{C}(-,C)}{I}\in \mathcal{T}$. We will to prove that

\[
\displaystyle\frac{\mathcal{C}(-,C')}{(I(-):h)}\subset \displaystyle\frac{\mathcal{C}(-,C)}{I}
\]
for all $h\in \mathcal{C}(C',C)$ and  $C'\in\mathcal{C}$. 
Indeed,  let $h\in \mathcal{C}(C',C)$, then  the natural morphisms $(-,h): \mathcal{C}(-,C')\longrightarrow \mathcal{C}(-,C)$ and 
$\pi : \mathcal{C}(-,C)\longrightarrow \displaystyle\frac{\mathcal{C}(-,C)}{I}$ induce the exact sequence

\[
0\longrightarrow (I(-):h)\longrightarrow \mathcal{C}(-,C')\stackrel{\pi h}\longrightarrow \displaystyle\frac{\mathcal{C}(-,C)}{I}
\]
which  implies $\displaystyle\frac{\mathcal{C}(-,C')}{(I(-):h)}\subset \displaystyle\frac{\mathcal{C}(-,C)}{I}$.
Hence $\displaystyle\frac{\mathcal{C}(-,C')}{(I(-):h)}\in\mathcal{T}$ since it is closed under subobjects and finally $(I(-):h)\in \mathcal{F}_{C'}.$ 
\end{proof}

Now we are ready for the first part of the main theorem for this section

\begin{theorem}
The maps $\mathcal{F}\rightarrow\mathcal{T}_{\mathcal{F}}$, $\mathcal{T}\rightarrow\mathcal{F}_{\mathcal{T}}$ induce  a bijection between heredity  pretorsion classes and  linear filters on $\mathcal{C}$.
\end{theorem}
\begin{proof}
 Starting with a linear filter $\mathcal{F}=\{\ \mathcal{F}_C\}$ we get 
 $\mathcal{T}=\{ M| \mathrm{Ann}(x,-)\in \mathcal{F}_C \text{ for all}\ x\in M(C)\  \text{ and each } C\in \mathcal{C}\}$ and we 
 obtain $\mathcal{I}=\{\mathcal{I}_C\}_{C\in\mathcal{C}}$ such that
 $\mathcal{I}_C=\{I\subset \mathcal{C}(-,C)|\displaystyle\frac{\mathcal{C}(-,C)}{I}\in \mathcal{T}\}$.
 
 We claim that
 $\mathcal{I}_C=\mathcal{F}_C$ for all $C\in \mathcal{C}$.  Indeed,
let $I\in \mathcal{I}_C$ then $\displaystyle\frac{\mathcal{C}(-,C)}{I}\in \mathcal{T}$, it follows that 
$\mathrm{Ann}(h+I(C),-)\in \mathcal{F}_C$ for all $h\in \mathcal{C}(C,C)$ 
and in particular $\mathrm{Ann}(1_C+I(C),-)\in \mathcal{F}_C$,
but $\mathrm{Ann}(1_C+I(h),-)=I$ by Lemma \ref{firstlemma}, hence $I\in \mathcal{F}_C$. 

Conversely, let  $F\in \mathcal{F}_C$. Then $\displaystyle\frac{\mathcal{C}(-,C)}{F}\in \mathcal{T}$.
Since $\mathrm{Ann}(g+F(B),-)=(F(-):g)\in \mathcal{F}_{B}$ for all  $B\in \mathcal{C}$ and each $ g\in\mathcal{C}(B,C)$ 
it follows that $F\in\mathcal{ I}_{C}$.

On the other hand if we start with the class $\mathcal{T}=\{M\}$ we first get 
$\mathcal{F}=\{\mathcal{F}_C\}_{C\in\mathcal{C}}$ such that
$\mathcal{F}_C=\{ I\subset \mathcal{C}(-,C)|\displaystyle\frac{\mathcal{C}(-,C)}{I}\in \mathcal{T}\}$ and we obtain 
$\mathcal{T'}= \{N| \mathrm{Ann}(x,-)\in \mathcal{F}_C \text{ for all}\ x\in N(C)\  \text{ and each } C\in \mathcal{C}\}$. 

We will prove $\mathcal{T}=\mathcal{T'}$.  Let $M\in \mathcal{T}$, then we put
\[
M\cong \coprod_{C\in \mathcal{C}}\coprod_{x\in M(C)}\displaystyle\frac{\mathcal{C}(-,C)}{I^x}
\]
by (\ref{isocop}),  it follows that $\displaystyle\frac{\mathcal{C}(-,C)}{I^x}\in \mathcal{T}$ for all $x\in F(C)$ since $\mathcal{T}$ is closed under subobjects, this implies that $\{I^x\}_{x\in M(C)}\subset \mathcal{F}_C$. 
Now, let $B\in \mathcal{C}$ and assume that  $\overline{X}=(\overline{X}_{\lambda})_{\lambda\in N}\in M(B)$ is non zero. Let  
$\{\overline{X}_{\lambda_{1}},\ldots,\overline{X}_{\lambda_{n}}\}$ be  the set of  all non-zero coordinates, we can put for $i=1,\ldots,n$:
\[
\overline{X}_{\lambda_{i}}=x_i+I^{x_i}(B), \text{}x_i\in \mathcal{C}(B,C_i)\ \text{ for some }C_i\in \mathcal{C}
 \]
Thus, $\mathrm{Ann}(\overline{X},-)=\cap_{i=1}^n \mathrm{Ann}(\overline{X}_{\lambda_i},-)=\cap_{i=1}^n \mathrm{Ann}(x_i+I^{x_i}(B),-)=
\cap_{i=1}^n(I^{x_i}(-):x_i)\in \mathcal{F}_{B}$ by $T_2$ and by Lemma \ref{firstlemma}. 
Hence,  $M\in \mathcal{T'}$. 

Now, asumme that   $M\in\{M|\mathrm{Ann}(x, -)\in \mathcal{F}_C \ \text{for\ all}\ x\in M(C), C\in \mathcal{C} \}$ 
and let  $\overline{X}\in M(C)$ whose  the unique non-zero coordinate is  $1_C+I^{x_C}(C)$,  then
\begin{eqnarray*}
\mathrm{Ann}(\overline{X})=\mathrm{Ann}(1_C+I^{x_C}(C),-)=I^{x_C}\in \mathcal{F}_C
\end{eqnarray*}
and this implies  $\displaystyle\frac{\mathcal{C}(-,C)}{I^{x_C}}\in \mathcal{T}$ since $\mathcal{T}$ is closed under subjects, finally 
 $M\cong \coprod_{C\in \mathcal{C}}\coprod_{x\in M(C)}\displaystyle\frac{\mathcal{C}(-,C)}{I^x}$ lies in $ \mathcal{T}$ since $\mathcal{T}$ is closed under coproducts.
\end{proof}
 Now we prove the main theorem of this section

\begin{theorem}
The maps $\mathcal{F}\rightarrow\mathcal{T}_{\mathcal{F}}$, $\mathcal{T}\rightarrow\mathcal{F}_{\mathcal{T}}$ induce  a bijection between heredity  torsion classes and  Gabriel filters on $\mathcal{C}$.
\end{theorem}
\begin{proof}
We will only prove that if $\mathcal{F}$ is a hereditary torsion class, the corresponding topology $\mathcal{F}=\{\mathcal{F}_C\}$, where
$\mathcal{F}_C=\{I\subset\mathcal{C}(-,C)| \displaystyle\frac{\mathcal{C}(-,C)}{I}\in \mathcal{T}\}$, satisfies $T_4$.  Indeed, 
let $I\subset\mathcal{C}(-,C)$ an ideal such that $(I(-):h)\in \mathcal{F}_{C'}$, for all $h\in J(C')$ and $C'\in \mathcal(C)$ 
for some $J\in \mathcal{F}_C$, and consider the exact sequence
\[
 0\longrightarrow \displaystyle\frac{J}{I\cap J}\longrightarrow  \displaystyle\frac{\mathcal{C}(-,C)}{I} 
 \longrightarrow \displaystyle\frac{\mathcal{C}(-,C)}{I+J} \longrightarrow 0,
\]
where $\displaystyle\frac{\mathcal{C}(-,C)}{I+J}\in \mathcal{T}$, since it is a quotient $\mathcal{C}$-modules 
of $\displaystyle\frac{\mathcal{C}(-,C)}{I}\in \mathcal{T}$. Now, we have $(I(-):h)=((I\cap J)(-):h)$ for all $h\in J(C')$,
 but this implies that $((I\cap J)(-): h)\in \mathcal{F}_{C'}$, and $\displaystyle\frac{\mathcal{C}(-,C)}{((I\cap J)(-): h)}\in \mathcal{F}$.

In this way we have a map of $\mathcal{C}$-modules for all $h\in J(C')$
\[
 \displaystyle\frac{\mathcal{C}(-,C)}{(I\cap J: h)}\stackrel{\varphi_h}\longrightarrow \displaystyle\frac{J}{I\cap J}
\]
given by $(\varphi_h)_{C''}(g+((I\cap J)(C''):h))=hg+(I\cap J)(C'')$,  which induces an epimorphism of $\mathcal{C}$-modules
\begin{eqnarray*}
 \varphi=(\varphi_h):\displaystyle\coprod_{h\in J(C')}\displaystyle\frac{\mathcal{C}(-,C')}{(I\cap J:h)}\longrightarrow
\displaystyle\frac{J}{I\cap J},
\end{eqnarray*}
given by $\varphi_{C''}((g_h+((I\cap J)(C''):h))_{h\in J(C')})=\sum_{h\in J(C')}(hg_h+(I\cap J)(C'))$ for all 
$g\in \mathcal{C}(C'',C')$

It follows that $\displaystyle\coprod_{h\in J(C')}\displaystyle\frac{\mathcal{C}(-,C')}{(I\cap J)(-):h}\in \mathcal{T}$. 
Since $\mathcal{T}$ is closed under coproducts and finally $\displaystyle\frac{J}{I\cap J}\in \mathcal{T}$ because 
$\mathcal{F}$ is closed under quotients. Since $\mathcal{T}$ is closed under extensions it follows that 
$\displaystyle\frac{\mathcal{C}(-,C)}{I}\in \mathcal{T}$ and therefore $I\in \mathcal{T}$.

Asumme  that $\mathcal{F}=\{\mathcal{F}_C\}_{C\in \mathcal{C}}$ is a Gabriel Filter for $\mathcal{C}$ and let 
$0\longrightarrow K \longrightarrow N \longrightarrow \displaystyle\frac{N}{K} \longrightarrow 0$ be an exact sequence 
of $\mathcal{C}$-modules for wich  $K$ and $\displaystyle\frac{N}{K}$ are in $\mathcal{T}$. 
Let be $C\in \mathcal{C}$ and $x\in N(C)$. Then $x+K(C)\in \displaystyle\frac{N}{K}(C)$ and $\mathrm{Ann}(x+K(C),-)\in \mathcal{F}_C$
since $\displaystyle\frac{N}{K}\in \mathcal{T}$. Let $C'\in \mathcal{C}$ and $h\in \mathrm{Ann}(x+K(C),C')\subset\mathcal{C}(C',C)$, 
we first will  prove the following:
\begin{itemize}
 \item[(1)]$\mathrm{Ann}(N(h)(x),-)\in \mathcal{F}_{C'}$
 \item[(2)]$(\mathrm{Ann}(x)(-):h)=\mathrm{Ann}(N(h)(x),-)$
\end{itemize}
1) Since $h\in \mathrm{Ann}(x+K(C),C')\subset \mathcal{C}(C',C)$ we have 
$0=\displaystyle(\frac{N}{K})(h)(x+K(C))=N(h)(x)+K(C')$, it follows
that $N(h)(x)\in K(C')$ and hence $\mathrm{Ann}(N(h)(x))\in \mathcal{F}_{C'}$, because $K\in\mathcal{T}$. 
2) Let  $C''\in \mathcal{C}$.  Then $f\in (\mathrm{Ann}(x,-):h)(C'')\subset \mathcal{C}(C'',C')$ if and only if 
$hf\in \mathrm{Ann}(x,C'')$ if and only if $N(hf)(x)=N(f)(N(h)(x))=0$ if and only if $f\in \mathrm{Ann}(N(h)(x),C'')$. It follows that
$(\mathrm{Ann}(x,-):h)=\mathrm{Ann}(N(h)(x),-)$ and (2) holds.

We proved that $(\mathrm{Ann}(x,-):h)=\mathrm{Ann}(N(h)(x),-)\in \mathcal{F}_{C'}$,
for all $h\in \mathrm{Ann}(x+K(C),C')$, $C'\in \mathcal{C}$. It follows that $\mathrm{Ann}(x,-)\in \mathcal{F}_C$ 
since $\mathrm{Ann}(x+K(C),-)\in \mathcal{F}_{C'}$, i.e. $N\in \mathcal{T}$.
\end{proof}

\section{Linear Topologies}

We have seen that a hereditary torsion theory $\mathcal{T}$  is characterized by the class $\mathcal{F_{T}}=\{\mathcal{F}_C\}_{C\in \mathcal{C}}$ such  that $\mathcal{F}_C=\{ I_{\lambda}\}_{\lambda\in \Lambda}$ is a family of right ideals  for which $\displaystyle\frac{\mathcal{C}(-,C)}{I_{\lambda}}\in \mathcal{T}$.
It turns out that for such a family $\mathcal{F}_C=\{ I_{\lambda}\}_{\lambda\in N}$, $C\in\mathcal{C}$ of right ideal
$I_{\lambda}\subset \mathcal{C}(-,C)$ we have that
$\mathcal{F}_C(C')=\{I_{\lambda}(C')\}$ is the family of neighbordhoods of the zero map $0:\mathcal{C'}\longrightarrow C$ for
certain topology on $\mathcal{C}(C',C)$. For this reason we start with a general review of topological groups.

Remember that an abelian group $G$ is a topological group if it is is equipped with a topology for wich
 the group operations
$(a,b) \longmapsto a+b$ and $a \longmapsto -a$ are continous functions $G\times G\longrightarrow G$ and $G\longrightarrow G.$

We have the well known 
\begin{proposition}
  Let $G$ be a topological group. For an  $a\in G$ the traslation map
\begin{eqnarray*}
 \mathcal{L}_a&:&G\rightarrow G\\
g&\mapsto&a+g
\end{eqnarray*}
is a homeomorfism.

\end{proposition}

So $U$ is a neighborhood of $a\in G$ if and only if $U-a$ is a neighbordhood of  $0$. Thus the topology of $G$ 
is complete by determined by a neighbordhood basis of 0.

\begin{proposition}
Let $\mathcal{C}$ a preadditive category and  $\mathcal{F}=\{\mathcal{F}_C\}_{C\in\mathcal{C}}$
 be a linear filter for $\mathcal{C}$. Then there exists  a
 topology $T_{(A,C)}$ for $\mathcal{C}(A,C)$, for each pair of objects $A,C\in \mathcal{C}$, 
 for wich the composition $\mathcal{C}(A,B)\times\mathcal{C}(A,B)\rightarrow \mathcal{C}(A,C)$ 
  and the sum $ \mathcal{C}(A,B)\times\mathcal{C}(A,B)\longmapsto \mathcal{C}(A,B)$,
$(f,g)\longmapsto (f+g)$, are continuos. Morever  $\mathcal{F}_C(A)$ is a basis of neighbordhoods for $0\in \mathcal{C}(A,C)$.
\end{proposition}
\begin{proof}
 We consider  a family of subsets $T_{(A,C)}\subset 2^{\mathcal{C}(A,C)}$ defined as follows:
$U\in T_{(A,C)}$ if $U$ is the empty set or if for each $x\in U$ there exist $I\in \mathcal{F}_C$ such that 
$x+I(A)\subset U$. 

a) $T_{(A,C)}$ is a topology for $\mathcal{C}(A,C)$.

1) It is clear that $\emptyset$ and $\mathcal{C}(A,C)$ are in $T_{(A,C)}$.
2) Let $U,V$ be in $T_{(A,B)}$ and $x\in U\cap V$. Then there are $I$ and $J$ in $\mathcal{F}_C$ such that $x+I(A)\subset U$ 
and $x+J(A)\in V$.
We have that $I\cap J \in \mathcal{F}_C$ since $\mathcal{F}$ is a linear filter of $\mathcal{C}$. 
If $y\in x+(I\cap J)(A)=x+I(A)\cap J(A)$ then there exist $r\in I(A)\cap J(A)$ such that $y=x+r\subset U\cap V$.
It follows that $x+(I\cap J)(A)\subset U\cap V$ and hence $U\cap V \in T_{(A,B)}$.
3) Let $\{U_{\lambda}\}_{\lambda\in N }$ be a family of sets such that $U_{\lambda}\in T$, 
and $x\in \bigcup$ $U_{\lambda \in N}$. Then $x\in U_\lambda$ for some
$\lambda\in N$ and there exists $I\in \mathcal{F}_C$ such that $x+I(A)\subset U_\lambda\subset \bigcup U_{\lambda \in N}$.

Now, if $U$ is a neighborhood of $0\in \mathcal{C}(A,C)$ then there exists 
$U'\in T_{(A,C)}$ for which $0\in U'\subset U$ and there exist $I\in \mathcal{F}_C$ such that $I(A)=0+I(A)\subset U' \subset U$,
it follows that $\mathcal{F}_C(A)$ is a basis of neighborhoods for $0$.

b)  For each pair $A,B\in\mathcal{C}$. The map $ \mathcal{C}(A,B)\times\mathcal{C}(A,B)\longmapsto \mathcal{C}(A,B)$,
$(f,g)\longmapsto (f+g)$ is continuos. 
Inded,  let $f,g\in \mathcal{C}(A,B)$ and $U\subset \mathcal{C}(A,B)$ open such that $f+g\in U$. Then, there
exists $I\in \mathcal{F}_B$ for wich $(f+g)+ I(A)\subset U$. 
Let $(r_1,r_2)\in (f+I(A),g+I(A))$, if we put  $r_1=f+t_1$, $r_2=g+t_2$, $t_1,t_2\in f+g+I(A)$, it follows that 
$r_1+r_2=f+g+t_1+t_2\in f+g+I(A)$.

c) $\mathcal{C}(A,B)\times \mathcal{C}(B,C)\longrightarrow \mathcal{C}(A,C)$.
$(f,g)\longmapsto (gf)$ is continuos. Indeed,  let $f\in\mathcal{C}(A,B)$, $g\in\mathcal{C}(B,C)$ and
$U\subset \mathcal{C}(A,C)$ open set such that $gf\in U$. Then,  there exits $I\in \mathcal{F}_C$, such 
that $gf+I(A)\subset U$. Now, we have  $(I(-):g)\in \mathcal{F}_B$, and if we take 
$(r_1,r_2)\in (f+(I(A):g))\times(g+I(A))\subset \mathcal{C}(A,B)\times \mathcal{C}(B,C)$, then  we can put 
$r_1=f+t_1$,  $t_1\in(I(A):g)$, and $r_2=g+t_2$,  $t_2\in I(A)$. It follows that 
\begin{eqnarray*}
r_2r_1&=&(g+t_2)(f+t_1)\\
          &=&gf+gt_1+t_2f+t_2t_1\in gf+ I(A)\subset \mathcal{C}(A,C)
\end{eqnarray*}
because $gt_1+t_2f+t_2t_1\in I(A)$.

\end{proof}

\section{Examples}

In this section we show two  examples, one of hereditary pretorsion theory and the another one of
hereditary torsion theory, both in the category of $\mathcal{C}$-modules.

\begin{definition}
Let be $\mathcal{U}$  a subclass of objects of $\mathrm{Mod}(\mathcal C)$. We say that an
$\mathcal{C}$-module $N$ is $\mathcal {U}$-sub-generated if $N$ is an subobject of another $\mathcal{U}$-generated object. 
We denote by $\sigma[\mathcal{U}]$ to the full subcategory of $\mathrm{Mod}(\mathcal{C})$ consisting of
all the $\mathcal{U}$-sub-generated objects.
\end{definition}

  It is clear that $\sigma[\mathcal{U}]$ is a pretorsion class for $\mathrm{Mod}(\mathcal{C})$, morever, we have the following.

\begin{proposition}
Let $\mathcal T$ a pretorsion class of $\mathrm{Mod}(\mathcal C)$. Then $\mathcal T$ is the category $\sigma [M]$, where
\[
M=\coprod\{\frac{\mathcal{C}(-,C)}{I}:\frac{\mathcal{C}(-,C)}{I}\text{is an object of }\mathcal{T}\}
\]
\end{proposition}
\begin{proof}
Since $\mathcal{T}$ is closed under submodules, coproducts and quotients it follows that $\sigma[M]$ is a full subcategory of $\mathcal{T}$.

Conversely, for every $\mathcal{C}$-module $N$ we have an isomoprphism 
\[
N\cong  \coprod_{C\in\mathcal{C}}\coprod_{x\in N(C)}\mathcal(-,C)/K^x
\]
by (\ref{isocop}). If $N$ is an object of $\mathcal{T}$, then we have that each $\mathcal{C}(-,C)/K^x$ is an subobject of $N$, it follows  $\mathcal{C}(-,C)/K^x\in\mathcal{T}$ since $\mathcal{T}$ is closed under subobjects, if follows immediately that $N$ is an object in $\sigma[M]$.
\end{proof}

\begin{theorem}
Let $I(-,?)$ be an bilateral ideal of $\mathcal{C}$ and $F=\coprod_{C\in\mathcal C} \mathcal{C}(-,C)/I(-,C)$, then  $\sigma[F]$ is the full fubcategory of $\mathrm{Mod}(\mathcal{C})$ consisting of all objects $N$ such that $IN=0$.  
\end{theorem}

\begin{proof}
Let $M$ be a $\mathcal{C}$-module $F$-generated. Then there exists an epimorphism $F^\Lambda\rightarrow M$. Then  we have
an epimorphism $I(F^\Lambda)\rightarrow I M$,  but it is easy to see that $I(F^\Lambda)=0$, then $IM=0$, it follows that $IN=0$ for all
$\mathcal{C}$ submodule  $N\subset M$.

Conversely, let $N$ be  a $\mathcal{C}$-module for which $IN=0$. 
Let $\eta=\{\eta_A\}_{A\in\mathcal{C}}:\mathcal{C}(-,C)\rightarrow N$ a natural transformation,
then we can  define a natural  transformation $\hat{\eta}=\{\hat{\eta}_B\}_{B\in\mathcal{C}}:\mathcal{C}(-,C)/I\rightarrow N$ as
\begin{eqnarray*}
\hat{\eta}_B&:&\mathcal{C}(B,C)/I(B)\rightarrow M(B)\\
&&f+I(B)\mapsto \eta_B(f).
\end{eqnarray*}

We claim that the natural transfomation $\hat{\eta}$ is well defined. Indeed, if $f+I(B)=g+I(B)$, then we have
$f-g\in I(B)\subset \mathcal{C}(B,C)$. Therefore there is  a commutative diagram
\[
\begin{diagram}
\node{\mathcal{C}(C,C)}\arrow{e,t}{\eta_C}\arrow{s,l}{\mathcal{C}(f-g,C)}
 \node{N(C)}\arrow{s,r}{N(f-g)}\\
 \node{\mathcal{C}(B,C)}\arrow{e,t}{\eta_B}
 \node{N(B)}
\end{diagram}
\]
It follows that $\eta_B(f)-\eta_B(g)=\eta_B(f-g)=\eta_B\mathcal{C}(f-g,C)(1_C)=N(f-g)(\eta_C(1_C))=0$ 
because $IN=0$. In this way if we proced as in (\ref{isocop}) we have
a well defined epimorphism  
\[
\coprod_{C\in\mathcal C}\coprod_{x\in N(C)}\frac{\mathcal{C}(-,C)}{I}\rightarrow N
\]
therefore $N$ is $F$-generated and also lies in $\sigma[F]$.
\end{proof}

Now, we show an example of hereditary torsion theory in the category of $\mathcal{C}$-modules.

Let $\mathbf{C}=\{C_\lambda\}_{\lambda\in\Lambda}$ a family  of objects in $\mathcal{C}$. The class  of $\mathcal{C}$-modules 
$\mathcal T=\{M|M(C_\lambda)=0\}$ is a hereditary torsion class.  Indeed, It is closed under subobjects, epimorphic images, and coproducts, and if
$0\rightarrow M\rightarrow L\rightarrow N\rightarrow 0$ is a exact sequence with $M,N\in\mathcal{T}$, then 
$0\rightarrow M(C_\lambda)\rightarrow L(C_\lambda)\rightarrow N(C_\lambda)\rightarrow 0$ is a exact sequence of  abelian
groups for all $C_{\lambda }\in\mathbf{C}$, but $M(C_\lambda)=M(C_\lambda)=0$ for all $\lambda\in\Lambda$ implies $L(C_\lambda)=0$ for all
$\lambda\in\Lambda$, therefore $ \mathcal{T}$ is closed under extensions an it is a hereditary torsion class.

We observe that $M\in\mathcal{T}$ if and only if $M(C_\lambda)=0$ for all $\lambda\in\Lambda$ if and only if
$\prod_{\lambda\in\Lambda}DM(C_\lambda)=0$ but we have
 $\prod_{\lambda\in\Lambda}DM(C_\lambda)=\prod_{\lambda\in\Lambda}(\mathcal C(C_{\lambda},-),DM)
 \cong\prod_{\lambda\in\Lambda}(M,D\mathcal{C}(C,-))\cong(M,\prod_{\lambda\in\Lambda}D\mathcal{C}(C_\lambda,-)$  by Yoneda's Lemma
 and Proposition \ref{dual}, it follows that the class $\mathcal{C}$ is cogenerated by an injective $\mathcal{C}$-module by
 Proposition \ref{injective}.
 
 We will prove that hereditary torsion classes are cogenerated by injective objects.
 
 By [AQM] we have that in the category $\mathrm{Mod}(\mathcal C)$ there exists injective envelopes. Thus  following the proof given in [Bo] it is easy 
 to see that the following proposition holds in the category of $\mathcal{C}$-modules. 
 
 \begin{proposition}[Bo, Prop. 3.2]\label{envelope}
 A torsion theory $(\mathcal{T},\mathcal F)$ is hereditary if and only if $\mathcal{F}$ is closed under injective envelopes.
 \end{proposition}
 
 Now, we can prove the following
 
 \begin{theorem}
 A torsion theory is hereditary if and only if it can be cogenerated by an injective module.
 \end{theorem}
 
\begin{proof}
the proof is very nearly from the given in [Bo] but we put it here for benefict to the reader.  Let $E$ be an injective modules and put
 $\mathcal{T}=\{M|\mathrm{Hom}(M,E)=0\}$.  If $M\in\mathcal{T}$ and $L$ is a $\mathcal{C}$-module of $M$ with a non-zero homomorphism
 $\alpha:L\rightarrow E$,  then extends to a homomorphism $M\rightarrow E$, which is impossible. Hence $L\in\mathcal{T}$, and the torsion theory cogenerated by $E$ is hereditary.
 
 Conversely, assume that $(\mathcal{T},\mathcal F)$ is a hereditary torsion theory. \ \ \ Put $E$$=\prod E(\mathcal{C}(-,C)/I)$
 with the product taken over the injective envelopes of all right 
 ideals $I\subset\mathcal{C}(-,C)$ such that $\mathcal{C}(-,C)/I\in \mathcal{F}$, for all $C\in\mathcal{C}$. 
 Then $E$ is a torsion-free module because is a product of  injective envelopes,  so $\mathrm{Hom}(M,E)=0$ for every 
 $M\in\mathcal{T}$. On the other hand, Let $M$ a $\mathcal{C}$-module and put it as
  $M=\coprod_{C\in\mathcal{C}}\coprod_{x\in M(C)}\mathcal{C}(-,C)/I^x$,  if $M\notin\mathcal{T}$ then there exist a direct sumand $\mathcal{C}(-,C_0)/I^{x_0}$  with a non-zero morphism $\alpha:\mathcal{C}(-,C_0)/I^{x_0}\rightarrow F$ for some $F\in\mathcal{F}$.  Otherwise,
   if $\mathrm{Hom}(\mathcal{C}(-,C)/I^x,\mathcal{F})=0$ for all pair $C,x$ 
 we should have  $\mathcal{C}(-,C)/I^x\in\mathcal{T}$ for all pair $C,x$ and $M$ would be in $\mathcal T$, since it is closed under coproducts.
 
 Now, if we put $F=\coprod_{B\in\mathcal{C}}\coprod_{y\in F(B)}\mathcal{C}(-,B)/I^y$ as in (\ref{isocop}), then  all the direct sumands $\mathcal{C}(-,B)/I^y$ lie en 
 $\mathcal{F}$ and so there is a monomorphism $F\rightarrow E$, so $ \alpha$ induces a morphism $\mathcal{C}(-,C_0)/I^{x_0}\rightarrow E$ wich can be extende to a non-zero map $M\rightarrow E$. It follows that $M\in\mathcal{T}$ if and only if $\mathrm{Hom}(M,E)=0$, i.e, $E$ conenerates the torsion theory.
  
\end{proof}
\subsection{Dense Ideals in path categories}

In this subsection we show an example of linear filters certains categories that appear  in the study of representations of dimesional
finite algebras.

\begin{definition}
An ideal $I(-,C)\subset\mathcal{C}(-,C)$ is dense  in $\mathcal{C}(-,C)$ if for all $B\in\mathcal{C}$ and $g\in\mathcal{C}(B,C)$,
 there exist $D\in\mathcal C$ 
and $h\in \mathcal{C}(D,B)$ shuch that $gh\in I(D,C)$. 

\end{definition}

So, we have   examples of linear filters for categories as the following proposition says.

\begin{proposition}
Let be $\mathcal{F}_C$ the family of all ideals in  $\mathcal{C}(-,C)$ that are dense. The collection $\mathcal{F}=\{\mathcal{F}_C\}_{C\in\mathcal{C}}$ is a linear filter for $\mathcal{C}$.
\end{proposition}
\begin{proof}
a) Assume that $I(-,C)\subset\mathcal C(-,C)$ is dense and $I\subset J$. Let $f\in\mathcal{C}(C',C)$ then  there  exists
a morphism $g:C''\rightarrow C'$ such that $fg \in I(C'',C)\subset J(C'',C)$, and $J$ is dense in $\mathcal C$.

b) Assume that $I(-,C), J(-,C)\subset\mathcal C(-,C)$ are dense. Let $f\in\mathcal{C}(C',C)$. Since 
$I$ is dense en $\mathcal C$ then  there  exists
a morphism $g:C_1''\rightarrow C'$ such that $fg \in I(C''_1,C)\subset\mathcal{C}(C_1'',C)$. Again, since $J$ is dense  there exists
$h:C_2''\rightarrow C_1''$ for wich $(fg)h\in J(C_2'',C)$. Therefore, if  $f\in\mathcal{C}(C',C)$ then
the existence of $gh:C_2''\rightarrow C'$ implies that $f(gh)\in (I\cap J)(C_2'',C)$.

c) Let $I\subset\mathcal{C}(-,A)$ a dense ideal and $h\in\mathcal{C}(C,A)$.  We will show that
$(I(-):h)\subset \mathcal{C}(-,C)$ is dense. Let $f\in\mathcal{C}(C',C)$, then $hf\in \mathcal{C}(C',A)$.
Since $I\subset \mathcal{C}(-,A)$ is dense there exist $C''\in\mathcal{C}$ and $g:C''\rightarrow C'$ for
wich $(hf)g\in I(C'')$. It is follows that $fg\in (I(C''):h)$.

\end{proof}

Now, we show two examples where dense ideals appear.

\textbf{Example 1.} If $K$ is a field and let $K(\mathbb{Z}A_{\infty},\sigma)$ be the mesh category (see [Rin] Sec. 2.1). 
 Indeed, we can think of  $K(\mathbb{Z}A_{\infty},\sigma)$ as an additive category 
 whose indescomposable objects the vertices and  the set of morphims between two indescomposable objects  is the
 $K$-vectorial space generated by the paths between them for which the squares commute. Let $f:A\rightarrow C$  be 
 a morphism. Then the ideal  $(f)$ generated by  $f$, defined as   $(f)(D)=\{fg|g\in\mathcal{C}(D,A)\}$, is dense in $\mathcal{C}(-,C)$.
 For this, let $g:B\rightarrow C$ a map
 and consider the set $\mathrm{Supp}(f)=\{X\in\mathcal{C}|\text{there exists a path }t:X\rightarrow A\}$, then clearly 
 there exists an object $D\in \mathrm{Supp}(f)$ and a path $h:D\rightarrow B$. Since
 $D\in \mathrm{Supp}(f)$ there exists a path $t':D\rightarrow A$. Then, by the mesh relations
 we have $gh=ft'\in (f)(D)$. This is, $(f)$  is dense in $\mathcal{C}(-,C)$.
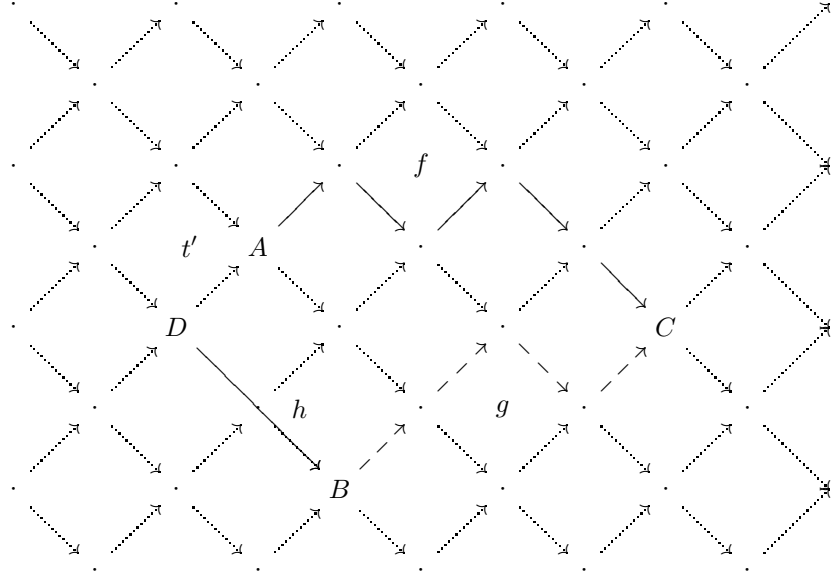
\[
\begin{diagram}
\node{\cdot}\arrow{se,t,..}{}
 \node{}
  \node{\cdot}\arrow{se,t,..}{}
 \node{}
   \node{\cdot}\arrow{se,t,..}{}
 \node{}
   \node{\cdot}\arrow{se,t,..}{}
 \node{}
   \node{\cdot}\arrow{se,t,..}{}
 \node{}\\
 \node{}
  \node{\cdot}\arrow{ne,t,..}{}\arrow{se,t,..}{}
    \node{}
  \node{\cdot}\arrow{ne,t,..}{}\arrow{se,t,..}{}
     \node{}
  \node{\cdot}\arrow{ne,t,..}{}\arrow{se,t,..}{}
     \node{}
  \node{\cdot}\arrow{ne,t,..}{}\arrow{se,t,..}{}
     \node{}
  \node{\cdot}\arrow{ne,t,..}{}\arrow{se,t,..}{}\\
  \node{\cdot}\arrow{se,t,..}{}\arrow{ne,t,..}{}
 \node{}
  \node{\cdot}\arrow{se,t,..}{}\arrow{ne,t,..}{}
 \node{}
   \node{\cdot}\arrow{se,t}{}\arrow{ne,t,..}{}
 \node{\Large{f}}
   \node{\cdot}\arrow{se,t}{}\arrow{ne,t,..}{}
 \node{}
   \node{\cdot}\arrow{se,t,..}{}\arrow{ne,t,..}{}
 \node{}\\
 \node{}
  \node{\cdot}\arrow{ne,t,..}{}\arrow{se,t,..}{}
    \node{\ \ \ t'}
  \node{A}\arrow{ne,t}{}\arrow{se,t,..}{}
     \node{}
  \node{\cdot}\arrow{ne,t}{}\arrow{se,t,..}{}
     \node{}
  \node{\cdot}\arrow{ne,t,..}{}\arrow{se,t}{}
     \node{}
  \node{\cdot}\arrow{ne,t,..}{}\arrow{se,t,..}{}\\
  \node{\cdot}\arrow{se,t,..}{}\arrow{ne,t,..}{}
 \node{}
  \node{D}\arrow[2]{se,t}{}\arrow{ne,t,..}{}
 \node{}
   \node{\cdot}\arrow{se,t,..}{}\arrow{ne,t,..}{}
 \node{}
   \node{\cdot}\arrow{se,t,--}{}\arrow{ne,t,..}{}
 \node{}
   \node{C}\arrow{se,t,..}{}\arrow{ne,t,..}{}
 \node{}\\
  \node{}
  \node{\cdot}\arrow{ne,t,..}{}\arrow{se,t,..}{}
    \node{}
  \node{\cdot}\arrow{ne,t,..}{}\arrow{se,t,..}{}
     \node{h\ \ \ \ \ \ \ \ \ }
  \node{\cdot}\arrow{ne,t,--}{}\arrow{se,t,..}{}
     \node{g}
  \node{\cdot}\arrow{ne,t,--}{}\arrow{se,t,..}{}
     \node{}
  \node{\cdot}\arrow{ne,t,..}{}\arrow{se,t,..}{}\\
   \node{\cdot}\arrow{se,t,..}{}\arrow{ne,t,..}{}
 \node{}
  \node{\cdot}\arrow{se,t,..}{}\arrow{ne,t,..}{}
 \node{}
   \node{B}\arrow{se,t,..}{}\arrow{ne,t,--}{}
 \node{}
   \node{\cdot}\arrow{se,t,..}{}\arrow{ne,t,..}{}
 \node{}
   \node{\cdot}\arrow{se,t,..}{}\arrow{ne,t,..}{}
 \node{}\\
 \node{}
  \node{\cdot}\arrow{ne,t,..}{}
    \node{}
  \node{\cdot}\arrow{ne,t,..}{}
     \node{}
  \node{\cdot}\arrow{ne,t,..}{}
     \node{}
  \node{\cdot}\arrow{ne,t,..}{}
     \node{}
  \node{\cdot}\arrow{ne,t,..}{}
 \end{diagram}
\]

\textbf{Example 2.}  Let  $(\mathfrak{T},\tau)\cong\mathbb{Z}A_{\infty}/(\tau^r)$ a stable tube of rank $r$ and  $B$ a complete set 
 of representants of the space of orbits of  $(\mathbb{Z}A_{\infty},\tau)$ (see [Rin, X.1]) . If  $\mathcal{B}$
 is the full subcategory of $K((\mathfrak{T},\tau))$ which objects are finite direct sums of objects in $B$. Then the ideal 
 $I_{\mathcal{B}}(-,C)$ is dense in $\mathbb{Z}A_{\infty}/(\tau^r)$. Indeed, if $f:X\rightarrow C$ is a map in
 $\mathbb{Z}A_{\infty}/(\tau^r)$ the we can write $f=(f_i)_{i=1}^n:\oplus_i^n X_i\rightarrow C$. Then, there exists a path
 $g_i:B_i\rightarrow X_i$, for $i=1,\ldots,n$. In this way the map $f(g_i)=(f_ig_i):B=\oplus B_i\rightarrow C$ lies in
 $I_{\mathcal{B}}(B,C)$, i.e $I_{\mathcal{B}}(-,C)$ is dense in $\mathbb{Z}A_{\infty}/(\tau^r)$. 

\begin{center}
\begin{tikzpicture}

\coordinate (A) at (2,.2);
\coordinate (B) at (3.5,.7);
\coordinate (C) at (2,2);
\coordinate (D) at (.5,2.6);
\coordinate (E) at (.5,3.6);

\foreach \coor/\formula in {A/{X_1},B/{X_2},C/{X_3},D/{X_4},E/{X_5}} {
  \fill (\coor) circle (2pt);
  \node[below right, inner xsep=-1ex] at (\coor) {$\formula$};
}

\coordinate (A) at (2,.2);
\coordinate (B) at (.5,.7);
\coordinate (C) at (0,2);
\coordinate (D) at (3.5,2.6);
\coordinate (E) at (3.5,3.6);

\foreach \coor/\formula in {A/{X_1},B/{B_2},C/{B_3},D/{B_4},E/{B_5}} {
  \fill (\coor) circle (2pt);
  \node[below right, inner xsep=-1ex] at (\coor) {$\formula$};
}

\draw[dashed,thin] (0,.5) to [out=315,in=225] (4,.5);
\draw[dashed,thin] (0,.5) to [out=45,in=135] (4,.5);

\draw[dashed,thin] (0,4) to [out=45,in=135] (4,4);
\draw[dashed,thin] (0,.5) to [out=90,in=270] (0,4);

\draw[dashed,thin] (4,.5) to [out=90,in=270] (4,4);

\draw[->] (0,3.5) to  [out=20,in=210] (.51,3.71);
\draw[->]  (3.51,3.71) to  [out=330,in=160](4,3.5);

\draw[dashed,thin] (0,4) to [out=340,in=200] (4,4);


\draw[ ->] (0,4) to  [out=320,in=160] (4,2.5);

\draw[->] (0,3) to  [out=320,in=160] (4,1.5);

\draw[->] (0,2) to  [out=320,in=160] (4,0.5);

\draw[->] (0,1) to  [out=320,in=160] (3,-.1);

\draw[->] (1,-.1) to  [out=20,in=205]  (2,.23);

\draw[->] (0,.5) to  [out=20,in=200] (2,1.1);
\draw[->] (2,1.1) to  [out=20,in=210] (3.5,1.7);
\draw[->] (0,.5) to  [out=20,in=210] (.51,.71);

\draw[->] (2,2.1) to  [out=20,in=210] (3.5,2.7);
\draw[->] (0,1.5) to  [out=20,in=210] (.51,1.71);

\draw[->] (0,2.5) to  [out=20,in=200] (2,3.1);
\draw[->] (2,2.1) to  [out=20,in=210] (3.5,2.7);
\draw[->] (0,2.5) to  [out=20,in=210] (.51,2.71);

\draw[->] (2,3.1) to  [out=20,in=210] (3.5,3.7);

\draw[->] (0.5,1.71) to  [out=18,in=200] (2,2.1);

\draw[->] (2,.23) to  [out=20,in=200] (3.5,.7);

\draw[->](3.5,.7) to  [out=20,in=210] (4,1);
\draw[->](3.5,1.7) to  [out=20,in=210] (4,2);
\draw[->](3.5,2.7) to  [out=20,in=210] (4,3);

\draw[dotted,->](4,1.5) to  [out=160,in=30] (0,1);
\draw[dotted,->](4,2.5) to  [out=160,in=30] (0,2);
\draw[dotted,->](4,3.5) to  [out=160,in=30] (0,3);

\draw[dotted,->](4,1) to  [out=160,in=350] (0,1.5);
\draw[dotted,->](4,2) to  [out=160,in=350] (0,2.5);
\draw[dotted,->](4,3) to  [out=160,in=350] (0,3.5);

\draw[dotted,->](4,1.5) to  [out=180,in=30] (0,.5);
\draw[dotted,->](4,2.5) to  [out=180,in=30] (0,1.5);
\draw[dotted,->](4,3.5) to  [out=180,in=30] (0,2.5);

\end{tikzpicture}
\end{center}

\end{document}